\documentclass[12pt,notitlepage]{article}
\usepackage{amssymb}
\usepackage{amsfonts}
\usepackage{a4wide}
\usepackage{amsmath,amsthm}
\usepackage{graphicx}

\newcommand{\To}{\mathbf{T}^2}
\newcommand{\R}{\mathbf{R}}
\newcommand{\Z}{\mathbf{Z}}
\newcommand{\N}{\mathbf{N}}

\newcommand{\ue}{\mathrm{e}}

\newcommand{\ud}{\mathrm{d}}

\newcommand{\cH}{\mathcal{H}}

\newcommand{\cF}{\mathcal{F}}

\newcommand{\la}{\langle}
\newcommand{\ra}{\rangle}

\newcommand{\Op}{\operatorname{Op_{\mathit N}}}
\newcommand{\tr}{\operatorname{Tr}}

\newcommand{\dist}{\operatorname{dist}}
\newtheorem{thm}{Theorem}
\newtheorem{prop}{Proposition}
\newtheorem{cor}{Corollary}
\newtheorem{Def}{Definition}

\newcommand{\bS}{\mathbb{S}}

\newenvironment{proofof}[1]{\noindent {\em Proof of #1.}}{ \hfill\qed\\ }

\begin{document}
\date{April 19, 2007}
\title{Quantisations of piecewise affine maps on the torus and their quantum
limits}
\author{Cheng-Hung Chang\footnote{NCTU Institute of Physics, 1001 Ta Hsueh Road, Hsinchu,
Taiwan 300, ROC} \and 
Tyll Kr{\"u}ger\footnote{Technische
Universit\"at Berlin, Institut f\"ur Mathematik,
Stra{\ss}e des 17.\ Juni 136, D-10623 Berlin, Germany} \and 
Roman Schubert\footnote{School of Mathematics, University of Bristol, University Walk,
Bristol BS8 1TW} \and 
Serge Troubetzkoy\footnote{Centre de physique th\'eorique,
F\'ed\'eration de Recherches des Unit\'es de Math\'ematiques de Marseille,
Institut de math\'ematiques de Luminy and 
Université de la M\'editerran\'ee,
Luminy, Case 907, F-13288 Marseille Cedex 9, France}}

\maketitle

\abstract{ For general quantum systems the semiclassical behaviour of 
eigenfunctions in relation to the ergodic properties of the underlying classical system 
is quite difficult to understand. The Wignerfunctions of eigenstates converge weakly to 
invariant measures of the classical system, the so called quantum limits,
 and one would like to understand which invariant measures 
can occur that way, thereby classifying the semiclassical behaviour of eigenfunctions. 

We introduce a class of maps on the torus for whose quantisations we 
can understand the set of quantum limits in great detail. In particular 
we can construct examples of ergodic maps which have singular ergodic 
measures as quantum limits, and examples of non-ergodic maps where 
arbitrary convex combinations of absolutely continuous ergodic measures 
can occur as quantum limits. 

The maps we quantise are obtained by cutting and stacking.}

\section{Introduction}

The correspondence principle in quantum mechanics states that in the 
semiclassical limit $\hbar\to 0$ classical mechanics emerges and 
governs quantum mechanical quantities for small de Broglie wavelength. 
One manifestation  of this principle is that the Wignerfunctions 
of eigenfunctions  converge weakly to invariant probability 
measures on phase space, 
the so called quantum limits. It is one of the big open problems 
in the field to classify the set of quantum limits, 
and it is in general not known which invariant measures can occur 
as quantum limits. In particular the case that the classical system is 
ergodic has attracted  a lot of attention. In this case the 
celebrated quantum ergodicity theorem, \cite{Shn74,Zel87,Col85},  states that almost all 
eigenfunctions have the ergodic Liouville measure as quantum limit, 
and one would like to know if in fact all eigenfunctions 
converge to the Liouville measure, i.e., if quantum unique ergodicity holds 
 or if there are exceptions. Possible candidates for exceptions 
would be  quantum limits 
concentrated on periodic orbits, a phenomenon called strong scaring. 
Another very interesting case is when the classical system 
is of mixed type, i.e., the phase space has several invariant 
components of positive measure, or, if there exist several invariant measures 
which are continuous relative to Liouville measure. Here the question is 
to what extend the quantum mechanical system respects  the splitting 
of the classical system into invariant components, i.e.., is a 
typical  quantum limit ergodic, or can every convex combination 
of invariant measures appear as a quantum limit. 

There has been recently considerable progress in some of these questions. 
For the cat map it was shown that quantum unique ergodicity does not hold, 
in \cite{deBFauNon03} a sequence of eigenfunctions was constructed 
whose quantum limit is a convex combination of the Liouville measure and 
an atomic measure supported on a periodic orbit. It was furthermore shown that 
the orbit can carry at most $1/2$ of the total mass of the measure 
\cite{BonDeB03,FauNon04}.  
The eigenvalues of the cat map have large multiplicities and 
this behaviour depends on the choice of the basis of eigenfunctions, 
in \cite{KurRud00} it was shown that for a so called Hecke 
basis of eigenfunctions 
quantum unique ergodicity actually holds.  

On compact Riemannian manifolds of negative curvature quantum 
unique ergodicity 
was conjectured in \cite{RudSar94} 
and for arithmetic manifolds it was recently 
proved by Lindenstrauss for Hecke bases of eigenfunctions, \cite{Lin06}. The non-arithmetic case 
is still open, but in \cite{Ana04,AnaNon06} the authors succeeded in proving lower bounds 
for the entropy of quantum limits on manifolds of negative curvature.

In this paper we introduce a class of model systems for which the 
set of quantum limits can be determined very precisely. 
This work was motivated by a paper of  Marklof and Rudnick were 
they gave an example of a quantum ergodic map which
one can prove to be quantum uniquely ergodic \cite{mr}. 
They mention that
there are no examples known where a quantum ergodic map is not quantum
uniquely ergodic. The purpose of our work was  to provide such examples,
in fact examples which are quite close in nature to those considered by
Marklof and Rudnick. The map they considered was a skew product map of the
torus $\To=\mathbf{R}^{2}/\mathbf{Z}^{2}$ of the form: 
\begin{equation*}
F:
\begin{pmatrix}
p \\ 
q
\end{pmatrix}
\mapsto 
\begin{pmatrix}
p+2q \\ 
f(q)
\end{pmatrix}
\mod1
\end{equation*}
where $f(p)$ is an irrational rotation of the circle $\mathbf{R}/\mathbf{Z}$. 
In this article we consider skew products of the same form for other
functions $f(p)$. In particular since a circle rotation is an interval
exchange transformation (IET for short) on two intervals, one of the
examples we will consider for $f(p)$ are interval exchanges on more
intervals. There are examples known of IETs which are not uniquely ergodic.
A consequence of our main result is that if $f(p)$ is an IET which is not
uniquely ergodic then each of the invariant measures of $F$ of the form
Lebesgue measure cross an invariant measure (with the exception of finitely
supported ones) of $f(p)$ is a quantum limit.

The plan of the paper is as follows. In Section \ref{sec:quantisation} 
we give a quick review of quantisation of maps on the torus, and introduce 
the maps we study and their quantisation. In particular we prove Egorov's theorem for 
these maps. In Section \ref{sec:q-limits} we turn our attention to quantum limits, we first 
give a general proof of quantum ergodicity for maps with singularities, and then show that 
for our particular class of maps the quantum limits can be understood purely in terms of 
the orbits of discretisations of the classical map. Then, in Section \ref{sec:cut-stack},
 we finally come to our main result. We first review the cutting and stacking construction 
to obtain maps and 
then show how it can be combined with discretisations to get a detailed understanding of quantum limits. 
In Theorem \ref{thm2} we summarise our main findings. Finally in the last two sections 
we discuss two examples and give some conclusions.

\section{Quantisation}
\label{sec:quantisation}

We give a short summary of the quantisation of maps on the torus, for more
details and background we refer to \cite{MG,DB}. 

\noindent \emph{The Hilbert space}: 
For $(p,q)\in \R^2$ we introduce the phase space translation operator 
\begin{equation*}
\mathrm{T}(p,q):=\mathrm{e}^{-\frac{\mathrm{i}}{\hbar }(q\hat{p}-p\hat{q})}\,\, ,
\end{equation*}
where $\hat{p}\psi (x):=\frac{\hbar }{\mathrm{i}}\psi ^{\prime }(x)$ and 
$\hat{q}\psi (x):=x\psi (x)$ for $\psi \in \mathcal{S}(\mathbf{R})$, are the
momentum and position operators, respectively. These operators 
are unitary on $L^2(\R)$ and satisfy for $(p,q),(p',q') \in \R^2$ 
\begin{equation}
\mathrm{T}(p+p^{\prime },q+q^{\prime })=\mathrm{e}^{-\frac{\mathrm{i}}{2\hbar 
}(qp^{\prime }-pq^{\prime })}\mathrm{T}(p,q)\mathrm{T}(p^{\prime },q^{\prime
})
\end{equation}
and they provide therefore a unitary irreducible representation of the
Heisenberg group on $L^{2}(\mathbf{R})$.

The state space of the classical map is obtained from 
$\R^2$ by identifying integer translates which gives the  two torus 
\begin{equation*}
\To=\mathbf{R}^{2}/\mathbf{Z}^{2}\,\,.
\end{equation*}
By mimicking this procedure the quantum mechanical  
state space is defined to be the space of distributions 
on $\R$ which satisfy 
\begin{equation*}
\mathrm{T}(1,0)\psi =\psi \,\,,\qquad T(0,1)\psi =\psi\,\, .
\end{equation*}
One finds that these two conditions  can only be fulfilled (for $\psi\neq const.$) if Planck's 
constant meets
the condition 
\begin{equation}
\frac{1}{2\pi \hbar }=N
\end{equation}
where $N$ is a positive integer. The allowed states then turn out to be
distributions of the form 
\begin{equation}
\psi (x)=\frac{1}{\sqrt{N}}\sum_{Q\in \mathbf{Z}}\Psi (Q)\delta 
\bigg(x-
\frac{Q}{N}\bigg)  \label{eq:states}
\end{equation}
with $\Psi (Q)$ a complex number satisfying 
\begin{equation*}
\Psi (Q+N)=\Psi (Q)\,\,.
\end{equation*}
So the $\Psi(Q)$ are functions on $\Z_N=\Z/N\Z$ and the space of these functions 
 will be denoted by $\mathcal{H}_{N}$, it is $N$
-dimensional and through the coefficients $\Psi (Q),Q=0,1,\cdots ,N-1$ it can
be identified with $\mathbf{C}^{N}$. There is a map $S_{N}:\mathcal{S}(%
\mathbf{R})\rightarrow \mathcal{H}_{N}$ defined by 
\begin{equation*}
S_{N}\psi :=\sum_{n,m\in \Z}(-1)^{Nnm}\mathrm{T}(n,m)\psi
\end{equation*}
which is onto. If we equip $\mathcal{H}_{N}$ furthermore with the inner
product 
\begin{equation*}
\langle \psi ,\phi \rangle _{N}:=\frac{1}{N}\sum_{Q\in \Z_N}\Psi ^{\ast
}(Q)\Phi (Q)
\end{equation*}
then $\mathcal{H}_{N}$ is a Hilbert space and $S_{N}$ is an isometry.

\noindent \emph{Observables}: In classical mechanics 
observables on the torus are given by functions on $\To$, these can be 
expanded into Fourier series 
\begin{equation*}
a=\sum_{n\in \mathbf{Z}^{2}}\hat{a}_{n}\ue(-\omega(z,n) )\,\, . 
\end{equation*}
where $z=(q,p)\in \To$, $\omega(z,n)=qn_2-pn_1$ and 
$\hat{a}_{n}:=\int_{\To}a(z)\mathrm{e}(\omega(z,n))\,\,
\mathrm{d}z$ denotes the $n$-th Fourier coefficient.  We use here and in the
following the notation $\mathrm{e}(x)=\mathrm{e}^{2\pi \mathrm{i}x}$ and $
\mathrm{e}_{N}(x):=\mathrm{e}^{\frac{2\pi \mathrm{i}}{N}x}$. These observables 
can be quantised by replacing $\mathrm{e}(-\omega(z,n) )$ by the 
the translation operator  
\begin{equation*}
\mathrm{T}_{N}(n):=\mathrm{T}\bigg(\frac{n_{1}}{N},\frac{n_{2}}{N}\bigg)
\end{equation*}
which acts on  $\mathcal{H}_{N}$. This is called Weyl quantisation, 
to a  classical observable $a\in C^{\infty }(%
\To)$ a corresponding quantum observable is  defined by 
\begin{equation}  \label{eq:def-weyl}
\Op[a]:=\sum_{n\in \mathbf{Z}^{2}}\hat{a}_{n}\mathrm{T}_{N}(n)\,\,,
\end{equation}
which is an operator on $\mathcal{H}_{N}$.  For example, if 
$a$ depends only on $q$ then the corresponding operator is just
multiplication with $a$, 
\begin{equation}
\label{eq:mult-op} \Op[a]\psi (q)=a(q)\psi (q)\,\,,
\end{equation}
and in terms of the coefficients $\Psi (Q)$ the action of $\Op[a]$
is given by 
$\Psi (Q)\mapsto a\big(Q/N\big)\Psi (Q)$.

The trace of a Weyl operator can be expressed in terms of the symbol, from 
\eqref{eq:def-weyl} follows easily, see \cite{MG,DB}, that for $a\in
C^{\infty}(\To)$ 
\begin{equation}
\lim_{N\rightarrow \infty }\frac{1}{N}\tr\Op[a]=
\int_{\To}a\,\,\mathrm{d}z\,\,.  \label{eq:szegoe}
\end{equation}

\noindent \emph{Quantisation of a map}:
Let 
\begin{equation*}
F:\To\to\To
\end{equation*}
be a volume preserving map. 
One calls a  sequence of unitary operators $U_{N} :\mathcal{H}_{N}\to\mathcal{H}_{N}$, $N\in \N$, 
a quantisation of the map $F$ 
if the correspondence principle holds, i.e., if for sufficiently nice 
functions $a$ one has 
\begin{equation}\label{eq:corr-principle}
U_{N}^{\ast }\Op[a]U_{N}\sim\Op[a\circ F]\,\,,
\end{equation}
for $N\to \infty$. If this relation holds, it is often called Egorov's theorem and 
it means that in the semiclassical limit, i.e., for $N\to\infty$, 
quantum evolution of observable approaches the classical time evolution.   

Let us now turn to the specific class of maps we want to quantise. 
They are  given by 
\begin{equation}
F:
\begin{pmatrix}
p \\ 
q
\end{pmatrix}
\mapsto 
\begin{pmatrix}
p+2q \\ 
f(q)
\end{pmatrix}
\mod1  \label{eq:def-of-F}
\end{equation}
where $f:[0,1]\rightarrow \lbrack 0,1]$ is a piecewise affine map given by a
cutting and stacking construction which we will describe in detail in
Section \ref{sec:cut-stack}. For the construction of the quantisation we only need 
the property that the singularity set $\mathbb{S}\subset [0,1]$ 
is nowhere dense. In
order to quantise this map we proceed similar to the construction in 
\cite{mr}, i.e., use a sequence of approximations to $f$. 
Consider the discretized interval 
\begin{equation}
D_{N}:=\{Q/N\,;Q\in \{0,1,2,\cdots ,N-1\}\}\,\,,
\end{equation}
i.e., the support of the Hilbert space elements \eqref{eq:states}. For each 
$N\in \N$ we will call a map $f_{N}:D_{N}\rightarrow D_{N}$ an approximation of $f$
if it is close to $f$ in a certain sense which we will now explain. Since $f$
is not assumed to be continuous we do not approximate it uniformly in the
supremum norm. Let $f(\mathbb{S}):=\{q:\exists q_{0}\in \mathbb{S}\ \text{such 
that}\, q=\lim_{q^{\prime }\rightarrow q_{0}}f(q^{\prime })\}.$  
We measure the difference between $f$ and an approximation $f_{N}$ only away
from the set $f(\mathbb{S})$. Let us call the relevant set 
\begin{equation*}
I_{\varepsilon }:=\{q\in \lbrack 0,1];\dist(q,f(\mathbb{S}))\geq
\varepsilon \}.
\end{equation*}
In the construction of $f_{N}$ in Section \ref{sec:cut-stack} we will choose
a sequence $\varepsilon _{N}$ with $\lim_{N\to\infty}\varepsilon _{N}=0$.  
For any fixed $\varepsilon _{N}$ the relevant measure
for the quality of the approximation will be 
\begin{equation}
\delta _{N}:=\delta _{N}\left( \varepsilon _{N}\right) :=\sup_{Q/N\in
I_{\varepsilon _{N}}}|f_{N}(Q/N)-f(Q/N)|\,\,.  \label{eq:def-of-delta}
\end{equation}
Any approximation $f_{N}$ then defines via \eqref{eq:def-of-F} an
approximation $F_{N}$ of $F$.

The  quantisation of $F$ is now defined to be the sequence of unitary operators 
\begin{equation}\label{eq:defof-quantization}
U_{N}\Psi (Q)=\mathrm{e}_{N}\big(-(\hat{f}_{N}^{-1}(Q))^{2}\big)\Psi (\hat{f}%
_{N}^{-1}(Q))\,\,,
\end{equation}
where $\hat{f}_{N}(Q):=Nf_{N}(Q/N)$ denotes the map induced by $f_{N}$ on $%
\mathbf{Z}_{N}=\mathbf{Z}/N\mathbf{Z}$. This is indeed a unitary operator on 
$\mathcal{H}_{N}$, with its adjoint given by 
\begin{equation*}
U_{N}^{\ast }\Psi (Q)=\mathrm{e}_{N}(Q^{2})\Psi (\hat{f}_{N}(Q))\,\,.
\end{equation*}

That this sequence of operators $U_{N}$ is really a quantisation of the map $%
F$ is the content of the Egorov theorem \eqref{eq:corr-principle} which we will now prove.     
In our case we have to be careful at the singularities of the map.
The singularities of $f$ and $F$ can be naturally identified,
thus without confusion we can denote by $\mathbb{S}$ the set of
singularities of $F$ as well. By $C_{\mathbb{S}}^{\infty }(\To)$ we
denote the space of functions in $C^{\infty }(\To)$ which vanish in a
neighbourhood of $F(\mathbb{S})$. We then find


\begin{thm}\label{thm:egorov}
For any $a\in C^{\infty }(\To)$ we have 
\begin{equation*}
U_{N}^{\ast }\Op[a]U_{N}=\Op[a\circ F_{N}]\,\,,
\end{equation*}
and for any $a\in C_{\mathbb{S}}^{\infty }(\To)$ 
there are constants $C(a),\varepsilon _{0}(a)>0$ such that for 
$\varepsilon _{N}<\varepsilon _{0}(a)$ 
\begin{equation*}
||U_{N}^{\ast }\Op[a]U_{N}-\Op[a\circ F]||\leq C(a)\,\delta
_{N}\,\,.
\end{equation*}
\end{thm}


\begin{proof}
The map $F$ and its quantisation $U_N$ can be decomposed into 
a product of two simpler maps and operators. Namely, with 
\begin{equation*}
F^{(1)}(p,q)=(p+2q,q)\,\, ,\qquad\text{and}\quad F^{(2)}(p,q)=(p,f(q))
\end{equation*}
we have 
\begin{equation*}
F=F^{(2)}\circ F^{(1)}\,\, .
\end{equation*}
These maps can be quantised separately as 
\begin{equation*}
U_N^{(1)}\Psi(Q):=\ue_N(-Q^2)\Psi(Q)\,\, ,\qquad \text{and}\quad 
U_N^{(2)}\Psi(Q):=\Psi(\hat{f}_N^{-1}(Q))\,\, ,
\end{equation*}
where $f_N$ denotes a discretisation of $f$ on the Heisenberg lattice.  
We then have 
\begin{equation*}
U_N=U_N^{(2)}U_N^{(1)}\,\, ,
\end{equation*}
and therefore it is sufficient to study the conjugation  
of an Weyl operator for the two operators separately. In the case of 
$U_N^{(1)}$ it is well known that Egorov's theorem is exactly fulfilled
\begin{equation*}
{U_N^{(1)}}^*\Op[a]U_N^{(1)}=\Op[a\circ F^{(1)}]\,\, ,
\end{equation*}
see \cite{mr}. For the study of the second operator we use that 
\begin{equation*}
\Op[a]U_N^{(2)}-U_N^{(2)}\Op[a]=0
\end{equation*} 
for any observable which is constant in $q$, therefore we can restrict 
ourselves
 in the following to the case that $a$ is constant in $p$. 
But then $\Op[a]$ is just multiplication with $a$, and we obtain 
\begin{equation*}
{U_N^{(2)}}^*\Op[a]U_N^{(2)}\psi(q) =a(f_N(q))\psi(q)\,\, .
\end{equation*}
For general observables we therefore obtain 
\begin{equation*}
{U_N^{(2)}}^*\Op[a]U_N^{(2)}=\Op[a\circ F^{(2)}]\,\, ,
\end{equation*}
and this proves the first part of the theorem. 

For the second part we have to estimate 
\begin{equation*}
||\Op[a\circ F]-\Op[a\circ F_N]||
\end{equation*}
and this can again be reduced to the case that 
$a$ depends only on $q$, and then 
$\Op[a\circ F]-\Op[a\circ F_N]$ is the multiplication operator 
with $a(f_N(q))-a(f(q))$. Since we have for a $b\in C^{\infty}(\To)$ 
which depends only on $q$ by \eqref{eq:mult-op} that 
\begin{equation*}
||\Op[b]||=\sup_{q\in D_N}|b(q)|\,\, ,
\end{equation*}
we obtain 
\begin{equation*}
||\Op[a\circ f]-\Op[a\circ f_N]||=
\sup_{q\in D_N}|(a\circ f_N-a\circ f)(q)|\,\, .
\end{equation*}
But for the right hand side we obtain by using 
\eqref{eq:def-of-delta} and that $a\equiv 0$ in a 
neighbourhood of $f\bS$ 
\begin{equation*}
\begin{split}
\sup_{q\in D_N}|(a\circ f_N-a\circ f)(q)|
&\leq\sup_{Q/N\in I_{\varepsilon_N}}|(a\circ f_N-a\circ f)(Q/N)|\\
&\hspace*{3cm}+ \sup_{Q/N\in I\backslash I_{\varepsilon_N}}|(a\circ f_N-a\circ f)(Q/N)|\\
 &\leq C(a)\delta_N\,\, .
\end{split}
\end{equation*}
since the second term on the right hand side is 
$0$ if $\varepsilon_N$ is small enough. 
\end{proof}

So if we can choose our approximations $f_{N}$ in a way that $\varepsilon
_{N}\rightarrow 0$ and $\delta_N \rightarrow 0$ for $N\rightarrow \infty $,
then the sequence of unitary operators $U_{N}$ reproduces the classical map $%
F$ in the semiclassical limit $N\rightarrow \infty $, and so the correspondence principle holds. 

\begin{Def}
A sequence of operators $U_{N}$ for which $\delta _{N}(\varepsilon _{N})$
and $\varepsilon _{N}$ tend to $0$ for $N\rightarrow \infty $ will be called
a \textbf{proper quantisation} of $F$
\end{Def}

The restriction on the support of the classical observables is necessary in
order that $a\circ F_{N}$ and $a\circ F$ are smooth for $N$ large enough.
For a general $a$ the composition $a\circ F$ is discontinuous which causes
problems with the Weyl quantisation. Theorem 1 is not valid without the
assumption on the singularities. This is shown by the following
counter-example.


\begin{prop}
Let $s\in \mathbb{S}$, $a(q)\in C^{\infty }(\To)$ depend only on $q$
with $a(s)\neq 0$ and let $g_{s}(q):=\sqrt{N}\mathrm{e}^{-(q-s)^{2}/N}$ be a
Gaussian centred at $s$ and $\psi _{s}:=S_{N}f_{s}$ be its projection to $%
\mathcal{H}_{N}$. Then there exists a constant $C$ such that 
\begin{equation*}
\lim_{N\rightarrow \infty }||\big(U_{N}^{\ast }\Op[a]U_{N}-\Op%
[a\circ F]\big)\psi _{s}||\geq C|a(s)|||\psi _{s}||
\end{equation*}
\end{prop}


\begin{proof}
Since $a$ depends only on $q$ the operator $U_N^*\Op[a]U_N- \Op[a\circ F]$ is 
given by multiplication with $a\circ f_N-a\circ f$, and we have 
\begin{equation}
||\big( U_N^*\Op[a]U_N- \Op[a\circ F]\big)\psi_s||^2
=\frac{1}{N}\sum_{n=1}^N \big|(a\circ f_N-a\circ f)(n/N)\big|^2 |\Psi_s(n/N)|^2\,\, ,
\end{equation}
where $\Psi_s(q)=\sum_{m\in\Z} g_s(q-m)$. Since $f$ is discontinuous at $s$ there exists an $\varepsilon>0$ and a $C>0$ such that 
\begin{equation}
\big|(a\circ f_N-a\circ f)(q)\big|\geq C|a(s)|\,\, ,\quad \text{for}\,\, q\in [s-\varepsilon,s+\varepsilon]\,\, .
\end{equation}
If we use now that $g_s$ is exponentially concentrated around $q=s$, which in particular implies 
\begin{equation}
\Psi_s(q)=g_s(q)+ O(\ue^{-c/N})\,\, ,\quad \text{for}\,\, q\in [s-\varepsilon,s+\varepsilon]\,\, ,
\end{equation}
we obtain 
\begin{equation}
\begin{split}
\frac{1}{N}\sum_{n=1}^N& \big|(a\circ f_N-a\circ f)(n/N)\big|^2 |\Psi_s(n/N)|^2\\
&=\frac{1}{N}\sum_{|n/N-s|\leq \varepsilon}\big|(a\circ f_N-a\circ f)(n/N)\big|^2 |\Psi_s(n/N)|^2\\
&\quad +
\frac{1}{N}\sum_{|n/N-s|>\varepsilon} \big|(a\circ f_N-a\circ f)(n/N)\big|^2 |\Psi_s(n/N)|^2\\
&\geq C^2 |a(s)|^2 ||\psi_s||^2+O(\ue^{-c/N})\,\, ,
\end{split}
\end{equation}
where we have furthermore used $ ||\psi_s||^2=\frac{1}{N}\sum_{|n/N-s|\leq \varepsilon} |\Psi_s(n/N)|^2
+O(\ue^{-c/N})$.
\end{proof}

We want to close this section with some comments about the underlying
motivation for the specific quantisation assumptions on the neighbourhood of
the singularities. Classically the singularities act like points with
infinite local expansion rate respectively Lyapunov exponent. Therefore any
perturbation in a small neighbourhood of the singularity set gives rise to 
an error
which becomes unbounded if the perturbation approaches the singularity set.
Since the quantised maps are a specific kind of perturbation it is natural
to leave the allowed error big for points close to the singularity set.

\section{Quantum limits and orbits}
\label{sec:q-limits}

We will now discuss the implications of the Theorem 1 for the eigenfunctions
of the quantised map. We will denote a orthonormal basis of eigenfunctions
of $U_{N}$ by $\psi _{k}^{N}$, $k=1,\cdots ,N$, 
\begin{equation*}
U_{N}\psi _{k}^{N}=\mathrm{e}_{N}(\theta _{k}^{N})\psi _{k}^{N}\,\,,
\end{equation*}
where we use the notation $\mathrm{e}_{N}(x)=\mathrm{e}^{\frac{2\pi \mathrm{i%
}}{N}x}$ and $\theta _{k}^{N}$ are the eigenphases. Each eigenfunction
defines a linear map on the algebra of observables 
\begin{equation*}
\Op[a]\mapsto \langle \psi _{k}^{N},\Op[a]\psi _{k}^{N}\rangle
\end{equation*}
and the leading term for $N\rightarrow \infty $ depends only on the
principal symbol $\sigma (a)$. The limit points of the sequence of all these
maps defined by the eigenfunctions define measures on the set of classical
observables and are called quantum limits (see, e.g., \cite{mr}). To put it more explicitly, a
measure $\nu $ on $\To$ is called a quantum limit of the system
defined by the $U_{N}$ if there exist a sequence of eigenfunctions $\{\psi
_{k_{j}}^{N_{j}}\}_{j\in \mathbf{N}}$ such that 
\begin{equation*}
\int_{\To}a(z)\,\,\mathrm{d}\nu =\lim_{j\rightarrow \infty }\langle
\psi _{k_{j}}^{N_{j}},\Op[a]\psi _{k_{j}}^{N_{j}}\rangle \,\,.
\end{equation*}
One of the major goals in quantum chaos, and more generally in semiclassical
analysis, is to determine all quantum limits that can occur and the relative
density of the corresponding subsequences of eigenfunctions. We say that a
subsequence of eigenfunctions $\{\psi _{k_{j}}^{N_{j}}\}_{j\in \mathbf{N}}$
has density $\alpha \in [ 0,1]$ if
\begin{equation}\label{eq:def-density}
\lim_{N\to\infty}\frac{\#\{k_j \,:\, N_j=N\}}{N}=\alpha\,\, ,
\end{equation}
provided that the limit exists.

Egorov's theorem usually implies that all quantum limits are invariant
measures for the classical map. In our case the same is true, but we have to
be careful at the singularities. If the set of singularities $\bS$ is nowhere dense then the space 
$C_{\mathbb{S}}^{\infty }(\To)$ in Theorem \ref{thm:egorov} is large enough so that as an immediate 
consequence  we
have:

\begin{cor}\label{cor:q-lim-inv}
Let us denote by $\mathcal{M}_{\mathrm{inv}}(F)$ the convex set of
F-invariant probability measures on $\To$, and by $\mathcal{M}_{%
\mathrm{qlim}}(U_{N})$ the set of quantum limits $\mu $ of $U_{N}$ with $\mu
(\mathbb{S})=0$, then 
\begin{equation*}
\mathcal{M}_{\mathrm{qlim}}(U_{N})\subset \mathcal{M}_{\mathrm{inv}}(F)
\end{equation*}
\end{cor}


So we only have to look at invariant measures as candidates for quantum
limits. In the simplest case that there is only one invariant probability
measure, i.e., that the system is uniquely ergodic, all eigenfunctions must
converge to this measure, and we have the so called unique quantum
ergodicity. This was the situation in the example of Marklof and Rudnick, \cite{mr}.

We will study now the relationship between properties of quantum limits and the 
density of subsequences of eigenfunctions converging to them more closely. 
Our first result gives an upper bound on the density. 


\begin{thm}\label{thm:upper-bound-density}
Let $U_{N}$ be a proper quantisation of $F$ and let $\mu $ be a quantum limit of $U_{N}$ with support $\Sigma
\subset \To$. Then any sequence of eigenfunctions which converge to 
$\mu $ has at most density $\mu _{\To}(\Sigma )$, where $\mu _{\To}$ is the Lebesgue measure on $\To$.
\end{thm}


\begin{proof}
Let $a_{\varepsilon}\in C^{\infty}(\To)$, $\varepsilon\in (0,1]$, be a sequence satisfying 
$a_{\varepsilon}|_{\Sigma}=1$ and $\lim_{\varepsilon\to 0} a_{\varepsilon}(z)=0$ for all 
$z\in \To\backslash \Sigma$, i.e., a sequence approximating the characteristic function of 
$\Sigma$.  If $\cF=\{\psi^{N_j}_{k_j}\}_{j\in\N}$ is a sequence of eigenfunctions with 
$\mu$ as quantum limit, and $\cF_N:=\{\psi^{N_j}_{k_j}\, ;\, N_j=N\}$, then 
\begin{equation}
\lim_{j\to\infty}\la \psi^{N_j}_{k_j},\Op[a_{\varepsilon}]\psi^{N_j}_{k_j}\ra=1
\end{equation}
and therefore 
\begin{equation}
\begin{split}
\alpha(\cF)&=\lim_{N\to\infty}\frac{1}{N}\sum_{\psi\in \cF_N}
 \la \psi ,\Op[a_{\varepsilon}]\psi \ra\\
&
\leq 
\lim_{N\to\infty}\frac{1}{N}\sum_{k=1}^N
 \la \psi^{N}_{k},\Op[a_{\varepsilon}]\psi^{N}_{k}\ra
=\int_{\To}a_{\varepsilon}\,\ud \mu_{\To}\,\, ,
\end{split}
\end{equation}
where we have used \eqref{eq:szegoe}. We now take the limit $\varepsilon\to 0$ and the 
theorem follows.  
\end{proof}

Since $\mu_{\To}(\bS)=0$ it follows in particular that a possible sequence of eigenfunctions
converging to a quantum limit concentrated on $\bS$ must have density $0$. 
This result is as well interesting for non-ergodic maps, because it gives an upper bound on the 
number of eigenfunctions whose quantum limits are supported on an invariant subset $\Sigma$ of $\To$
by the volume of $\Sigma$. 

In case that the system is ergodic, we can actually determine the quantum limit of 
most eigenfunctions.


\begin{thm}\label{thm:qe}
Let $U_{N}$ be a proper quantisation of $F$ and assume that $\mu _{\To}$ is ergodic. Then there exists a
subsequence of eigenfunctions of density one which converges to $\mu _{\To}$.
\end{thm}


This is the usual quantum ergodicity result, but  our
proof differs from the standard one (see e.g. \cite{MG}) in that we rely on
the convexity definition of ergodicity, this is more convenient when dealing
with maps with singularities as has been observed in \cite{GL}. Recall that 
$\mu _{\To}$ is ergodic if it is extremal in the convex set of
invariant probability measures, i.e., if $\mu _{\To}=\alpha \mu
_{1}+(1-\alpha )\mu _{2}$ with $\mu _{2}\neq \mu _{\To}$ then $\alpha
=1$ and $\mu _{1}=\mu _{\To}$.

\begin{proof}
The existence of a subsequence $\cF=\{\psi^{N_j}_{k_j}\}_{j\in\N}$ of
 density one of eigenfunctions with quantum limit $\mu_{\To}$  is equivalent to 
\begin{equation}
\lim_{N\to\infty}\frac{1}{N}\sum_{k=1}^{N}
 |\la \psi^{N}_{k},\Op[a]\psi^{N}_{k}\ra -\overline{a}|=0\,\, ,
\end{equation}
see \cite{MG}. We first observe that by ergodicity every subsequence $\cF=\{\psi^{N_j}_{k_j}\}_{j\in\N}$ 
of positive density satisfies 
\begin{equation}\label{eq:mean-subsequence}
\lim_{N\to\infty}\frac{1}{|\cF_N|}\sum_{\psi\in\cF_N}\la \psi ,\Op[a]\psi \ra
=\overline{a}\,\, .
\end{equation}
To see this we consider the sequence   
\begin{equation}
a_N:=\frac{1}{|\cF_N|}\sum_{\psi\in\cF_N}\la \psi ,\Op[a]\psi \ra\,\, ,
\end{equation}
this is a bounded sequence since $\Op[a]$ is bounded, and 
therefore there exists a convergent subsequence 
$\{a_{N_j}\}_{j\in\N}$. Now using \eqref{eq:szegoe} we have with a convergent subsequence 
\begin{equation}\label{eq:sequence-decomp}
\begin{split}
\overline{a}&=\lim_{j\to\infty}\frac{1}{N_j}\sum_{k=1}^{N_j} \la \psi^{N_j}_{k},\Op[a]\psi^{N_j}_{k}\ra\\
&=\lim_{j\to\infty}\frac{|\cF_{N_j}|}{N_j}\frac{1}{|\cF_{N_j}|}\sum_{\psi\in\cF_{N_j}}\la \psi,\Op[a]\psi \ra\\
&\quad +\lim_{j\to\infty}\frac{N_j-|\cF_{N_j}|}{N_j}\frac{1}{N_j-|\cF_{N_j}|}\sum_{\psi_k^{N_j}\in \cH_{N_j}\backslash\cF_{N_j}}\la \psi^{N_j}_{k},\Op[a]\psi^{N_j}_{k}\ra\\
&= \alpha(\cF)\int a\,\, \ud \mu_1+(1-\alpha(\cF))\int a\,\, \ud \mu_2
\end{split}
\end{equation}
where $\mu_1$ and $\mu_2$ are invariant measures defined by 
\begin{align}
\lim_{j\to\infty}
\frac{1}{|\cF_{N_j}|}\sum_{\psi\in\cF_{N_j}}\la \psi,\Op[a]\psi \ra
&=\int a\,\, \ud \mu_1\\
\lim_{N_j\to\infty}\frac{1}{N_j-|\cF_{N_j}|}\sum_{\psi_k^{N_j}\in \cH_{N_j}\backslash\cF_{N_j}}\la \psi^{N_j}_{k},\Op[a]\psi^{N_j}_{k}\ra
&=\int a\,\, \ud \mu_2
\end{align}
These two measures exist by the assumption that the subsequence $\{a_{N_j}\}_{j\in\N}$ 
was convergent, and they are invariant by Theorem \ref{thm:egorov}. 
But equation \eqref{eq:sequence-decomp} can be rewritten as 
\begin{equation}
\mu=\alpha \mu_1+(1-\alpha)\mu_2
\end{equation}
and if $\mu$ is ergodic and $\alpha\neq 0$ this is only possible if $\mu_1=\mu$, 
and this proves that 
\begin{equation}
\lim_{j\to\infty} a_{N_j}=\overline{a}\,\, .
\end{equation}
Since this holds for every convergent subsequence of $\{a_N\}_{N\in\N}$ 
$\overline{a}$ is the only limit point and 
\eqref{eq:mean-subsequence} follows.  

Now assume that 
\begin{equation}
\lim_{N\to\infty}\frac{1}{N}\sum_{k=1}^{N}
 |\la \psi^{N}_{k},\Op[a]\psi^{N}_{k}\ra -\overline{a}|=
C>0\,\, ,
\end{equation}
then there must either exists a subsequence $\{ k_j\}_{j\in \N}$ 
of positive density with  
\begin{equation}
\la \psi^{N_j}_{k_j},\Op[a]\psi^{N_j}_{k_j}\ra-\overline{a}\geq C/2
\end{equation}
or one with 
\begin{equation}
\la \psi^{N_j}_{k_j},\Op[a]\psi^{N_j}_{k_j}\ra-\overline{a}\leq -C/2\,\, .
\end{equation}
But the mean value 
of the sequence 
$\la \psi^{N_j}_{k_j},\Op[a]\psi^{N_j}_{k_j}\ra-\overline{a}$ 
must tend to $0$ by \eqref{eq:mean-subsequence} and so we 
have a contradiction if $C\neq 0$. 
\end{proof}

The previous results, 
Corollary \ref{cor:q-lim-inv}, Theorem \ref{thm:upper-bound-density} and Theorem \ref{thm:qe}, 
are quite general, they are valid for all quantised maps which satisfy Egorov's theorem. 
We now turn to a more concrete study of the eigenfunctions for the specific quantum maps 
\eqref{eq:defof-quantization}. Our aim is to
show that the quantum limits are determined by the spatial distribution of
the periodic orbits of the discretisation of the classical map. The
eigenvalue equation 
\begin{equation*}
U_{N}\psi =\mathrm{e}_{N}(\theta )\psi
\end{equation*}
leads to the following recursion equation for $\psi $ 
\begin{equation}\label{eq:recursion}
\Psi (\hat{f}_{N}(Q))=\mathrm{e}_{N}(\theta -Q^{2})\Psi (Q)\,\,.
\end{equation}
From this recursion relation we obtain 
\begin{equation*}
|\Psi (\hat{f}_{N}(Q))|^{2}=|\Psi (Q)|^{2}\,\,,
\end{equation*}
and this implies that the probability densities in position space defined by the eigenfunctions are
invariant under the map $f_{N}$. In order to determine these densities it is
therefore sufficient to determine the spatial distribution of the orbits of 
$f_{N}$.

For the further investigation we note that each periodic orbit of $f_{N}$
carries at least one eigenfunction. And we can determine the eigenfunctions
and eigenvalues more explicitly, let $\mathcal{O}$ be an periodic orbit of
period $|\mathcal{O}|=K$, then the recursion relation \eqref{eq:recursion} gives 
\begin{equation*}
\Psi (Q)=\mathrm{e}_{N}\bigg(K\theta -\sum_{k=0}^{K-1}\big[\hat{f}
_{N}^{k}(Q)]^{2}\bigg)\Psi (Q)\,\,.
\end{equation*}
So if $\psi $ should be an eigenfunction with eigenvalue $\mathrm{e}%
_{N}(\theta )$ we get the condition 
\begin{equation*}
K\theta -\sum_{k=0}^{K-1}\big[\hat{f}_{N}^{k}(Q)]^{2}=Nm
\end{equation*}
with $m\in \{0,1,\cdots ,K-1\}$. This determines the eigenvalues, and then
the corresponding eigenfunctions follow from the recursion relation and the
normalisation condition. Summarising we get:


\begin{prop}
Let $\mathcal{O}$ be an orbit of period $K=|\mathcal{O}|$ of $f_{N}$, then
there exists $K$ orthogonal eigenfunctions of $U_{N}$ with support $\mathcal{%
O}$. The eigenphases are given by 
\begin{equation*}
\theta _{k}=S_{\mathcal{O}}+\frac{N}{K}\,k
\end{equation*}
with $k\in \{0,1,\cdots ,K-1\}$ and 
\begin{equation*}
S_{\mathcal{O}}:=\frac{1}{K}\sum_{k^{\prime }=0}^{K-1}\big[\hat{f}%
_{N}^{k^{\prime }}(Q)]^{2}\,\,,
\end{equation*}
and a normalised eigenfunction corresponding to $\theta _{k}$ is given by 
\begin{equation*}
\Psi _{k}(\hat{f}_{N}^{k^{\prime }}(Q_{0}))=\mathrm{e}_{N}\bigg(k^{\prime
}\theta _{k}-\sum_{m=0}^{k^{\prime }}[\hat{f}_{N}^{m}(Q_{0})]^{2}\bigg) %
\left( \frac{N}{K}\right) ^{1/2}
\end{equation*}
where $Q_{0}\in \mathcal{O}$ is an arbitrary point on the orbit and $%
k^{\prime }\in \{0,1,\cdots ,K-1\}$.
\end{prop}


The quantum lattice $D_N$ of $N$ points is a disjoint union of all periodic
orbits of $f_N$, and on each of these orbits are as many eigenfunctions
concentrated as the orbit is long. But that means that the orbits determine the 
quantum limits and the relative density of the corresponding sequence of
eigenfunctions.

To each periodic orbit $\mathcal{O}$ we can associate a probability measure
on $[0,1]$ 
\begin{equation}  \label{eq:deltaO}
\delta_{\mathcal{O}}(q):=\frac{1}{|\mathcal{O}|}\sum_{Q\in \mathcal{O}%
}\delta \bigg(q-\frac{Q}{N}\bigg)
\end{equation}
which is invariant under $f_N$.

\begin{cor}
Let $\mathcal{O}_{j}^{(N)}$, $j=1,\ldots ,J_{N}$ be the periodic orbits of $%
f_{N}$ and let $\delta _{j}^{(N)}$ be the corresponding probability measures %
\eqref{eq:deltaO}. Assume that there is an invariant measure $\nu $ of $f$
and a sequence of periodic orbits $\{\mathcal{O}_{j_{k}}^{(N_{k})}\}_{k\in 
\mathbf{N}}$ such that 
\begin{equation*}
\lim_{k\rightarrow \infty }\delta _{j_{k}}^{(N_{k})}=\nu \,\,,
\end{equation*}
then $\nu $ is a quantum limit of $U_{N}$.
\end{cor}

In all our examples the sequence $\left( N_{k}\right) $ contains all natural
numbers. Thus there are two possible definitions of the density of a
sequence of periodic orbits, $\mathcal{G}=\{\mathcal{O}_{j_{k}}^{(N_{k})}\}_{k\in 
\mathbf{N}}$, 
\begin{equation*}
\alpha (\mathcal{G})=\lim_{N\rightarrow \infty }\frac{|
\mathcal{O}_{j_{N}}^{(N)}|}{N}\qquad\text{or }\quad\beta (\mathcal{O}_{j_{k}}^{(N_{k})})
=\lim_{N\rightarrow \infty }\frac{1}{N}
\sum\limits_{k:N_{k}=N}|\mathcal{O}_{j_{k}}^{(N_{k})}|\,\, ,
\end{equation*}%
whenever the limit exists, which we will call the $\alpha$-density
or $\beta$-density of $\mathcal{G}$, respectively.

This corollary suggests that the set of quantum limits coincides with the
set of limits points of the sequence of orbit measures $\delta _{j}^{(N)}$
and that moreover the relative densities of the convergent subsequences
coincide too. This is  true if there are no multiplicities in the
eigenvalues. If there are eigenvalues of multiplicity larger than one, then
the eigenspace can mix the contribution of the different orbits. But even in this case 
there always  exists a choice of a basis of eigenfunctions corresponding to the orbit  
measures $\delta _{j}^{(N)}$.  Notice that in this case the 
$\beta$-density of the sequence of orbits  coincides with the density of the corresponding sequence of 
eigenfunctions defined in \eqref{eq:def-density}.

\section{Cutting and Stacking constructions}

\label{sec:cut-stack}

Cutting and stacking is a popular method in ergodic theory to construct maps
on the interval which are isomorphic models of arbitrary measure preserving
dynamical systems. The construction gives a piecewise isometric mapping on
the interval with Lebesgue measure as an invariant measure. One can also
think of this transformation as a countable interval exchange transformation.

The final mapping will be defined only Lebesgue almost everywhere. None the
less we can use this model to study certain other invariant measures which
are well behaved with respect to the cutting and stacking construction.

For a very readable introduction into cutting and stacking see the recent
book by Shields \cite{shields} or the old book by Friedman \cite{friedman}.
We will now give a short description of the basic construction scheme and
the relevant definitions.

A stack (or column) $S$ is a finite family of enumerated disjoint intervals $%
\{I_{j}\}_{j=1}^{h(S)}$ where $h(S)$ is called the height of $S$. The $I_{j}$
are subintervals of $[0,1]$ of equal length which is called the width of $S$. 
There are two possible conventions: either one can take all the intervals to
be open, or all the intervals to be closed on the left and open on the
right. Which convention we choose is not important for this paper.
The intervals $I_{1}$ and $I_{h(s)}$ are called the bottom and top of 
$S$ respectively. We define a transformation $f_{S}$ as follows, if the
point $x\in I_{j}$ is not on the top of $S$ and not on the boundary of $%
I_{j} $ then it gets mapped to the point directly above it (see Figure
1(a)). Since $I_{j+1}$ and $I_{j}$ have equal width, $f_{S}$ is simply the
canonical identification map between $I_{j}$ and $I_{j+1}$. Interpreted in $%
[0,1]$ this means that $f_{S}:I_{j}\rightarrow I_{j+1}$ such that $%
x\rightarrow x+\partial ^{-}I_{j+1}-\partial ^{-}I_{j}$ where $\partial ^{-}$
denotes the left boundary point of an interval. The construction clearly
defines $f^{-1}$ on all $I_{j}$ except at the bottom.

\begin{figure}[t]
\begin{center}
\includegraphics[width=14cm]{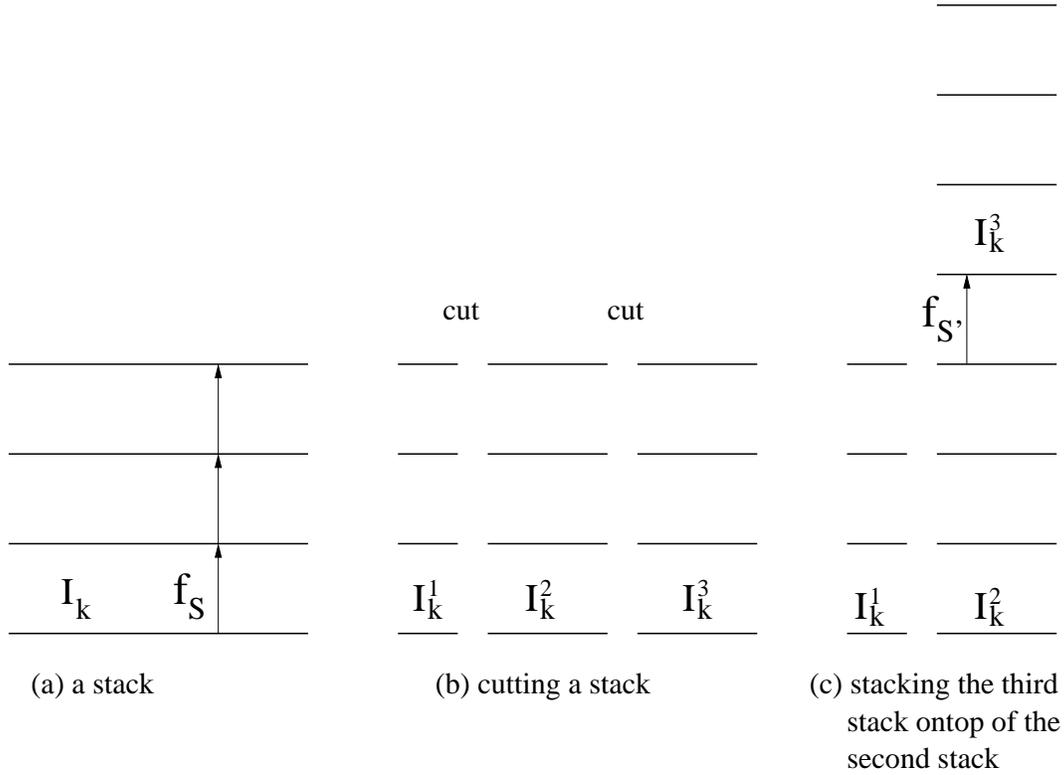}
\parbox{12cm}{\caption{ \footnotesize The cutting and stacking construction. In a given stack (a) the
mapping $f_{S}$ is defined, except at the top interval. In (b) the stack is
cut into three substacks, and in (c) the third substack is stacked onto the
second one. This gives an extension $f_{S^{\prime }}$ of the map $f_{S}$
which was not defined on the top of substack two before.} }
\label{fig:cuttingstacking}
\end{center}
\end{figure}

A stack family $\mathcal{S}$ is a finite or countable set of stacks $%
\{S_{i}\}=\{\{I_{j}^{i}\}_{j=1}^{h(S_{i})}\}$ such that all $I_{j}^{i}$ are
disjoint and $\cup I_{j}^{i}=[0,1]$. In this paper we will work only with
finite stack families. On $\mathcal{S}$ one defines a transformation $f_{%
\mathcal{S}}$ by $f_{\mathcal{S}}|_{S_{i}}=f_{S_{i}}$ except on the
collection of top intervals.

Given two stacks $S_{i}$ and $S_{j}$, $i\neq j$ of equal width one can
define a new stack $S^{\prime }$ by stacking $S_{j}$ on $S_{i}$ that is
\begin{eqnarray*}
S^{\prime } &=&\{I_{k}^{\prime }\}_{k=1}^{h(S_{i})+h(S_{j})} \\
\;\;\;\;I_{k}^{\prime } &=&I_{k}^{i}\text{ for }k\leq h(S_{i})\text{ and }%
I_{k}^{\prime }=I_{k-h(S_{i})}^{j}\text{ for }k>h(S_{i}).
\end{eqnarray*}

Correspondingly one gets a new transformation $f_{S^{\prime }}$ which agrees
with $f_{S_{i}}$ and $f_{S_{j}}$ except on $I^{i}_{h(S_{i})}$, where $%
f_{S_i} $ was not defined before.

It remains to define the cutting of stacks. A cutting of a stack $%
S=\{I_{k}\} $ is a splitting of $S$ into two (or more) disjoint stacks $%
S_{1} $and $S_{2}$ with intervals $\{I_{k}^{1}\}$ and $\{I_{k}^{2}\}$ such
that 
\begin{equation*}
I_{k}^{1}\cup I_{k}^{2}=I_{k}\;\text{ and }\partial ^{-}I_{k}^{1}<\partial
^{+}I_{k}^{1}=\partial ^{-}I_{k}^{2}<\partial ^{+}I_{k}^{2}\;\;\forall k
\end{equation*}%
that is $I_{k}^{1}$ is always the left component of the partition of $I_{k}$
into $I_{k}^{1}$ and $I_{k}^{2}$ (Figure 1(b)). The definition of $%
f_{\{S_{1},S_{2}\}}$ is as above. Multiple cutting of $S$ is defined
analogously.

A stack family $\mathcal{S}(n)$ is obtained from a stack family $\mathcal{S}%
(n-1)$ by cutting and stacking, if each $S_{i}(n)$ from $\mathcal{S}(n)$ can
be obtained by successive cuttings and stackings of stacks from $\mathcal{S}%
(n-1)$. By construction $f_{\mathcal{S}(n)}$ is an extension of $f_{\mathcal{%
S}(n-1)}$. If one has a sequence $\{\mathcal{S}(n)\}_{n\geq 1}$ of stack
families such that each $\mathcal{S}(n)$ is obtained from $\mathcal{S}(n-1)$
by cutting and stacking and furthermore

\begin{equation}
\lim_{n\rightarrow \infty }\sum_{S_{k}(n)\in \mathcal{S}(n)}\text{width}%
(S_{k}(n))=0,  \label{eq:width}
\end{equation}%
then $\lim f_{\mathcal{S}(n)}=f$ is an invertible transformation on $[0,1]$
defined everywhere except at a set of zero Lebesgue measure. Note that $f$
is always aperiodic.

The ``partition'' of $[0,1]$ into the intervals of $\mathcal{S}(1)$, 
\footnote{%
Here we can ignore the boundary points of the interval since $f$ is not
defined on them.} the starting object of the construction, gives a natural
symbolic dynamics for $f$. The coding is unique for all points whose
infinite orbit is defined.

For convenience, we denote the intervals of $\mathcal{S}(n)$ by $%
I_{j}^{i}(n) $ where $i$ indexes the stack and $j$ the interval in the
stack. We consider the set

\begin{equation*}
\mathbb{S}=\overline{\big(\cup _{i,n}\partial I_{h(S_{i})}^{i}(n)\big)\cup \big(\cup
_{i,n}\partial I_{1}^{i}(n)\big)}.
\end{equation*}

The set $\mathbb{S=S}\left( f\right) $ is called the singularity set of the
map $f$. It consists of all the points of discontinuity of $f$ and all the
points where the map $f$ or $f^{-1}$ is not defined. The boundary points of
the intervals $I_{j}^{i}$ which are not top or bottom intervals are not
included in this set, the map is defined and continuous on such points!
Furthermore $\cup _{i,j,n}\partial I_{j}^{i}(n)$ is included in the set $%
\cup _{k\in \mathbb{Z}}f^{k}\mathbb{S}$, where $f^{k}\mathbb{S}%
:=\{x:f^{-k}x\in \mathbb{S}\}$.

Historically the cutting and stacking construction was invented to represent
the dynamics with respect to a single invariant measure as a countable
interval exchange transformation with the canonical invariant Lebesgue
measure. The construction is universal in the sense that for every
measurable dynamical system $\left( M,g,\mu \right) $ one can explicitly
give a cutting and stacking representation $\left( \left[ 0,1\right]
,f\left( g\right) ,\mu _{L}\right) $ \cite{shields}. The following
proposition which will be needed for the application of Theorem 3 seemed to
be unknown.

\begin{prop}
\label{prop measure} Let $\left( \Sigma ,\sigma \right) $ be a symbolic
dynamical system over a finite or countable alphabet and let $\mu $ be a
shift-invariant measure such that $\mu \left( \left[ w\right] \right) >0$
for all cylinder-sets corresponding to finite words $w$ appearing in $\Sigma $%
. Then the associated cutting and stacking construction yields a
representation $\left( \left[ 0,1\right] ,f,\mu _{L}\right) $ of $\left(
\Sigma ,\sigma ,\mu \right) $ such that all nonatomic invariant measures $%
\nu ^{\ast }$ on $\left( \Sigma ,\sigma \right) $ have a corresponding
isomorphic invariant measure $\nu $ on $\left( \left[ 0,1\right] ,f\right) $
with the property $\nu \left( \mathbb{S}\left( f\right) \right) =0.$
\end{prop}

\begin{proof} By definition all finite symbolic words have a
representation as orbit segments of the cutting and stacking construction.
Furthermore to each nonperiodic symbol sequence corresponds a unique point
in $\left[ 0,1\right] $ whose orbit under $f$ is well defined and commutes
with the shift. By the Poincar\'{e} recurrence theorem every $f-$ invariant
measure is supported on recurrent points. Since the singular orbits,
respectively symbolic sequences, are the ones which eventually or
asymptotically fall onto the singularity set they do not intersect with the
recurrent nonperiodic symbol sequences. Therefore all invariant measures-
except the finitely supported ones- give zero measure to the singularity set.
\end{proof}

\subsection{The approximating family}

Each $f$ defined by cutting and stacking provides us with a natural
approximation family $\{f_{\mathcal{S}(n)}\}$ which we will use now to
define the approximation mappings on the rational points $D_{N}=\{\frac{Q}{N}%
:\,Q\in \{0,.....,N-1\}\}$ for the quantisation. Let the points in $%
G_{i,j}(n,N):=D_{N}\cap I_{j}^{i}(n)$ be enumerated from left to right and
let
\begin{equation*}
K(N,I_{j}^{i}(n)):=\sharp G_{i,j}(n,N)\text{ and }K(N,S_{i}(n)):=\min_{j}%
\{K(N,I_{j}^{i}(n))\}.
\end{equation*}

$K(N,S_{i}(n))$ is just the smallest number of points from the
discretisation in an interval in the stack $S_{i}(n)$. Let $\hat{G}%
_{i,j}(n,N)$ be the set of the first $K(N,S_{i}(n))$ points from $%
G_{i,j}(n,N)$ denoted by $\{x_{e}^{j,i}\}_{e=1}^{K(N,S_{i}(n))}$.

We define $f_{N,n}$ first on $\hat{D}_{N,n}=\cup _{i,j}\hat{G}_{i,j}(n,N)$
by setting 
\begin{equation}
f_{N,n}x_{e}^{j,i}=x_{e}^{j+1,i}\text{for }j<h_{i}(n):=h\left(
S_{i}(n)\right) .  \label{map up}
\end{equation}%
We call these the internal orbit segments. Clearly $\left\vert
f_{N,n}x_{e}^{j,i}-fx_{e}^{j,i}\right\vert =O\left( \frac{1}{N}\right) .$


\begin{figure}[t]
\begin{center}
\includegraphics[width=14cm]{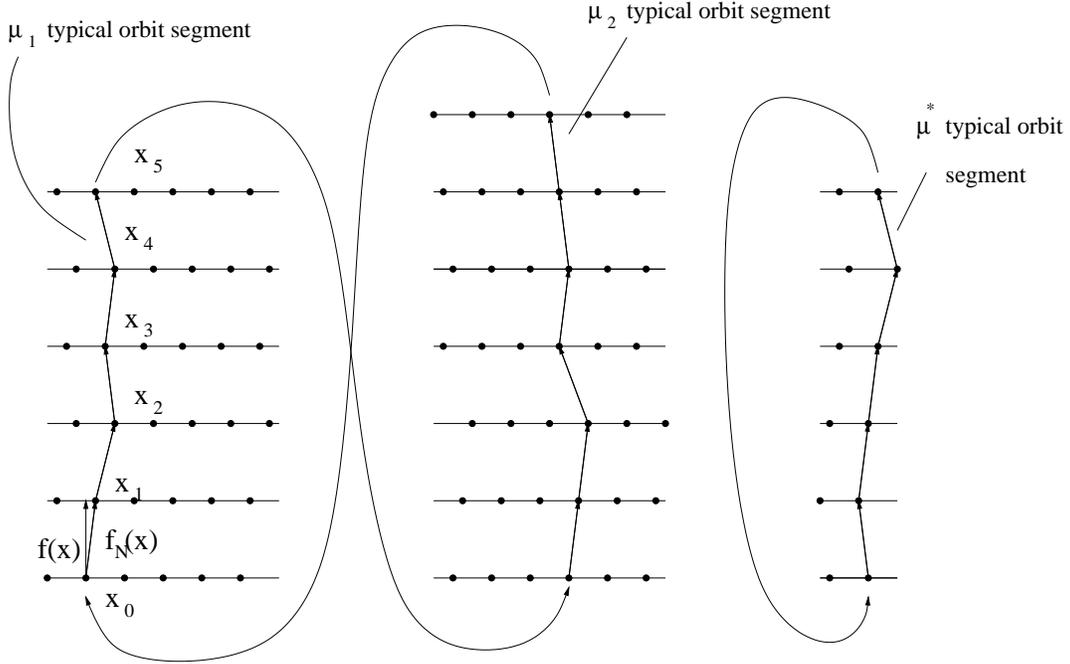}
\parbox{12cm}{\caption{\footnotesize The approximate mapping $f_{N}$ on the discrete set $D_{N}$
(denoted by the full dots), and the way the orbit segments from different
stacks are concatenated in order to produce a quantum limit of the form $%
\protect\alpha _{1}\protect\mu _{1}+\protect\alpha _{2}\protect\mu _{2}$. A
singular limit $\protect\mu ^{\ast }$ is obtained by a stack of small width
and where the orbits are concatenated from top to the bottom. }}
\end{center}
\end{figure}
We call each approximation mapping $f_{N,n}$ on $D_{N}$ whose restriction to 
$\hat{D}_{N,n}$ is given by the above construction an ergodic approximation.

Let $\check{D}_{N,n}$ be the set of points not in $\hat{D}_{N,n}$ and not in
in any of the top intervals $I_{h(S_{i})}^{i}(n).$ For $x\in \check{D}_{N,n}$
let $f_{N,n}x$ be the closest point to $fx.$ Note that $f_{N,n}$ is not
necessarily an invertible map, thus the construction implies that $%
\max\limits_{x\in \check{D}_{N,n}}|f_{N,n}x-fx|=O\left( \frac{1}{N}\right) $%
. It is clear that for fixed $n$ we have

\begin{equation}
\lim_{N\rightarrow \infty }\frac{\sharp \hat{D}_{N,n}}{\sharp D_{N}}=1.
\label{dense}
\end{equation}%
To complete the definition of $f_{N,n},$ it remains to define the mapping of
the points $\hat{G}_{i,h_{i}(n)}(n,N)$ on the tops of the stacks to points $%
\hat{G}_{i,1}(n,N)$ on the bottoms of the stacks. This will be done in a way
to produce periodic orbits which approximately mimic a given invariant
measure. Furthermore it remains to link $N$ to a given $n$ to get a good
approximation. In essence we have to require that for each fixed stage $n$
construction we have enough discretisation points in each stack. That means
that with the increase of $N$ we pass from $n$ to $n+1$ only as we pass a
critical threshold value $N_{n}.$ This can be done already without a precise
description of the gluing between top and bottom of the stacks. We need that
the approximation family $f_{N}$ is good enough to apply Theorem 1 for a
sequence $\varepsilon _{N}\rightarrow 0$ such that $\delta _{N}\left(
\varepsilon _{N}\right) \rightarrow 0$. Note that Theorem 1 does not impose
any requirements on the rate of convergence. The basic idea is to keep $N$
large enough compared with $n$ such that all intervals $I_{j}^{i}(n)$ of $%
\mathcal{S}(n)$ contain sufficiently many points from $D_{N}$. Let $%
b_{n}=\min\limits_{i}width(S_{i}(n))$. Choose a function $n(N)$ going to
infinity such that

\begin{equation}
\lim_{N\rightarrow \infty }\min_{i}\,K(N,S_{i}(n(N)))=\infty  \label{width}
\end{equation}%
which is equivalent to require $\frac{1}{b_{n}}=o(N)$. Furthermore let $%
\varepsilon _{N}=\max\limits_{i}width(S_{i}(n))$. Equation \eqref{eq:width}
implies that $\varepsilon _{N}$ tends to $0$ since $n(N)$ tends to infinity.
With this choice of $\varepsilon _{N}$ the points where $f_{N,n\left(
N\right) }$ is not yet defined do not contribute to $\delta _{N}(\varepsilon
_{N}),$ hence one obtains $\delta _{N}(\varepsilon _{N})=O(\frac{1}{N})$. An
approximation family $f_{N}:=f_{N,n(N)}$ for any function $n(N)$ satisfying
the above requirements is called proper.

We say that a measure $\mu $ appears as a quantum limit if one can find a
proper approximating family $f_{N}$ and associated quantisation $U_{N}$ such
that $\mu $ is a quantum limit of $U_{N}$. The notion of quantum limit as
well as the notion of density where introduced in Corollary 2.

\begin{thm}
\label{thm2} a) \ If $\mu $ is an absolute continuous ergodic measure for $f$%
, then $\mu $ appears as a positive $\alpha $ and $\beta -$ 
density quantum
limit. Furthermore the quantum limit has full $\alpha $ and $\beta -$
density if $\mu $ is the Lebesgue measure.

b) \ If $\mu $ is a nonatomic, singular ergodic measure for $f$ then $\mu $
appears as a quantum limit. Furthermore the quantum limit must have zero 
$\alpha $ and $\beta -$ 
density.

c) \ If $\mu _{1}$ and $\mu _{2}$ are two absolute continuous ergodic
invariant measures, then $\alpha _{1}\mu _{1}+\alpha _{2}\mu _{2}$ appears
as a positive 
$\alpha $ and $\beta -$ 
density quantum limit for any $\alpha
_{1},\alpha _{2}\in \left( 0,1\right) $ with $\alpha _{1}+\alpha _{2}=1$. 

d) \ If $\mu _{1}$ and $\mu _{2}$ are two ergodic measures at least one of
which is singular, then $\alpha _{1}\mu _{1}+\alpha _{2}\mu _{2}$ appears 
as a quantum limit for any $\alpha _{1},\alpha _{2}\in \left( 0,1\right) $
with $\alpha _{1}+\alpha _{2}=1$. Furthermore the quantum limit must have
zero 
$\alpha $ and $\beta -$ 
density.
\end{thm}

A transformation $f:[0,1]\rightarrow \lbrack 0,1]$ is called a finite rank
transformation if one can construct it via cutting and stacking such that
the number of stacks $\#\{S_{i}(n)\}$ in the $n$-th stack family $\mathcal{S}%
(n)$ is bounded (independent of $n$).

\begin{cor}
\label{cor:finiterank} If $f$ is a finite rank transformation, then every
ergodic non atomic invariant measure appears as the quantum limit of any
proper ergodic approximation family.
\end{cor}

\begin{proof}
For a finite rank transformation, the singularity sent $\mathbb{S}$ is a
countable set which has only a finite number of points of density. Thus any
non atomic invariant measure can not be support on $\mathbb{S}$.
\end{proof}

A point $x$ is called $\mu $-typical if 
${\lim_{m,u\rightarrow \infty }}\frac{1}{m+u+1}\sum\limits_{i=-u}^{m}\delta
\left( f^{i}x\right) =\mu $. For the proof of the theorem we need the
following simple fact whose proof is omitted since it is immediate from the
definition of weak convergence of measures.

\begin{prop}
\label{prop typical} Fix $x$ be $\mu $-typical. Let $j(m,u)$ and $%
\varepsilon (j)$ be functions such that $j(m,u)\rightarrow \infty $ for $%
m,u\rightarrow \infty $ and $\varepsilon (j)\rightarrow 0$ as $j\rightarrow
\infty $. Let $\{y_{k}^{(j)}\}_{k\in \mathbb{Z}}$ be a family of sequences
with the property that $\left\vert y_{k}^{(j)}-f^{k}x\right\vert \leq
\varepsilon (j)$ for $-u\leq k\leq m$. Then ${\lim_{m,u\rightarrow \infty }}%
\frac{1}{m+u+1}\sum\limits_{k=-u}^{m}\delta \left( y_{k}^{(j(m,u))}\right)
=\mu $.
\end{prop}

\begin{proofof}{Theorem 3}
For the proof of part b) and part a) for $\alpha -$ density we complete the
definition of $f_{N}$ as follows. Let us complete each internal orbit
segment into a periodic orbit by setting $%
f_{N}x_{e}^{h_{i}(n),i}=x_{e}^{1,i} $ (compare (\ref{map up})) . For points $%
x\in D_{N}\setminus \mathbb{S}$ where the map is not yet defined we have the
freedom to map $x$ anywhere, for preciseness define $f_{N}x$ to be the
closest point to $fx.$ By Corollary 2 it is enough to show that there is a
sequence of periodic orbits on $D_{N}$ whose point-mass average converge for 
$N\rightarrow \infty $ to the considered measure $\mu $.

Fix an arbitrary enumeration $\left\{ \mathcal{O}_{j}^{\left( N\right)
}\right\} $ of the periodic orbits on $D_{N}$. For $\mu $ absolute
continuous the set of points $x\in \left[ 0,1\right] $ with ${%
\lim_{m,u\rightarrow \infty }}\frac{1}{m+u+1}\sum\limits_{i=-u}^{m}\delta
\left( f^{i}x\right) =\mu $ has positive Lebesgue measure and in the case $%
\mu $ is the Lebesgue measure it has full measure. Let $x$ be $\mu $-typical
and consider for each $N$ the stack $S_{i}(n(N))$ in which $x$ is placed,
where $n(N)$ is a function satisfying the requirements of Equation %
\eqref{width}. Let $\mathcal{O}_{j_{\mathit{l}}\text{ }}^{\left( N\right) }$
be the set of periodic orbits in $\cup {j}\hat{G}_{i,j}(n(N),N).$ They stay 
$width(S_{i}(n(N)))$-close to the orbit segment $%
\{f^{-m_{0}}x,.....,f^{n_{0}}x\}$ where $m_{0}$ and $n_{0}$ are the smallest
and largest iterates such that $f_{S_{i}(n(N))}^{k}x$ is still defined. Note
that $m_{0}+n_{0}+1=h\left( S_{i}(n(N)\right) )$. Since 
\begin{equation*}
width(S_{i}(n(N)))\leq \varepsilon
_{N}:=\max\limits_{i}width(S_{i}(n(N)))\rightarrow 0\ \ \text{for }%
N\rightarrow \infty
\end{equation*}%
we can apply the above proposition to the family $\mathcal{O}_{j_{\mathit{l}}%
\text{ }}^{\left( N\right) }$ (with fixed $\mathit{l}$)\textit{\ }to
conclude that $\mu $ is a quantum limit, however this construction has not
yet proved the positive density.

To prove the positive density we need a quantified version of the above. Let 
$\mathcal{J}_{q}$ be the set of subintervals of $[0,1]$ with boundary points
of the form $\frac{p}{q}$. For the convergence of a sequence to a measure it
is clearly enough to check the characteristic function averages with respect
to the elements of $\cup _{q}\mathcal{J}_{q}$. A stack $S_{i}(n(N))$ is
called $\varepsilon -q-$good with respect to $\mu $ if for all $x\in
I_{1}^{i}(n(N))$ (i.e.~$x$ in the base of the stack $S_{i}(n(N))$) 
\begin{equation*}
\mu (J)-\varepsilon \leq \frac{1}{h_{i}(n(N))}\sum\limits_{i=0}^{h_{i}(n)}%
\mathbf{1}_{J}\left( f|_{S_{i}(n(N))}^{i}(x)\right) \leq \mu (J)+\varepsilon
\ for\ \forall J\in \mathcal{J}_{q}
\end{equation*}%
where $h_{i}(n(N))$ denotes the height of the stack $S_{i}(n(N))$. Denote
the family of such stacks by $\mathcal{G}(n,q,\varepsilon ,\mu )$ and by $%
\mathbb{G}(n,q,\varepsilon ,\mu )$ the set of points contained in $\mathcal{G%
}(n,q,\varepsilon ,\mu )$. Clearly one has $\forall q,\varepsilon >0$ 
\begin{eqnarray*}
\ \ \lim\limits_{n\rightarrow \infty }\ \mu _{L}(\mathbb{G}(n,q,\varepsilon
,\mu )) &=&\text{ }\mu _{L}\left( x:x\text{ is }\mu -\text{ typical}\right) 
\text{and} \\
\text{ }\lim\limits_{n\rightarrow \infty }\mu (\mathbb{G}(n,q,\varepsilon
,\mu )) &=&1.
\end{eqnarray*}%
Let $\varepsilon (n)$ be a sufficiently slowly decreasing function and $%
q(\varepsilon (n))$ be a sufficiently slowly increasing function that 
\begin{equation}
\lim\limits_{n\rightarrow \infty }\mu _{L}(\mathbb{G}(n,q(n),\varepsilon
(n),\mu ))=\mu _{L}\left( x:x\text{ is }\mu -\text{ typical}\right) .
\label{density}
\end{equation}

With this new notation we are ready to prove that $\mu $ is a quantum limit
of positive density. We define the sequence $\mathcal{O}_{j_{\mathit{l}}%
\text{ }}^{\left( N\right) }$of periodic orbits which give rise to the
desired quantum limit as follows. For fixed $N$ the set $\mathbb{G}%
(n(N),q(n(N)),\varepsilon (n(N)),\mu )\cap \hat{D}_{N,n(N)}$ consists of a
collection of points of periodic orbits. The sequence $\mathcal{O}_{j_{%
\mathit{l}}\text{ }}^{\left( N\right) }$consists of the set of these orbits.
The positive $\beta -$ density and full density in the case of Lebesgue
measure then follows from Equation \eqref{density} and \ref{dense}. To prove
part a) for the $\alpha -$ density we only need to modify the map $f_{N}$ on
top of the stacks. This will be done such that the collection of periodic
orbits $\mathcal{O}_{j_{\mathit{l}}\text{ }}^{\left( N\right) }$ becomes
just one periodic orbit for each $N.$ This completes the proof of part a).

To prove part b) observe that due to Proposition \ref{prop measure} we have $%
\mu \left( \mathbb{S}\right) =0$ and hence can apply Theorem 1 to get an
invariant measure out of the quantum-limit. One considers the set of $\mu $%
-typical points. From the proof of Proposition \ref{prop measure} follows
that the orbit of every $\mu $-typical point does not intersect nor converge
to the singularity set $\mathbb{S}$. Since the set of $\mu -$typical points
has zero Lebesgue measure we can obtain only a zero density quantum limit
just as in the proof of part a).

To prove part c) one has to modify the construction of the approximating
mapping $f_{N}$ in the following way. Instead of making $f_{N}$ periodic
within each stack $S_{i}(n(N))$ we want to connect two stacks say $%
S_{i}(n(N))$ and $S_{j}(n(N))$ where the orbit segments in the $i-$ th stack
respectively $j$-th stack are approximately typical for $\mu _{1}$
respectively $\mu _{2}$ to get an average of $\mu _{1}$ and $\mu _{2}$.

For $l=1,2$ let $A_{l}(N):=\{i:S_{i}(n(N))\in \mathcal{G}(n\left( N\right)
,q(n),\varepsilon (n),\mu _{l})$. On $\ \hat{D}_{N,n(N)}$ define $f_{N}$ as
before by $f_{N}x_{e}^{j,i}=x_{e}^{j+1,i}$ for $j<h_{i}(n)$. For $%
x_{e}^{j,i}\in \mathcal{G}(n\left( N\right) ,q(n),\varepsilon (n),\mu
_{l})\cap \hat{D}_{N,n(N)}$ one has for $\forall J\in \mathcal{J}_{q(n)}$ 
\begin{equation*}
\mu _{l}(J)-\varepsilon (n)+O(\frac{1}{N})\leq \frac{1}{h_{i}(n)}%
\sum\limits_{0\leq k\leq h_{i}(n)-1}\mathbf{1}_{J}\left(
f_{N}^{k}x_{e}^{1,i}\right) \leq \mu _{l}(J)+\varepsilon (n)+O(\frac{1}{N}).
\end{equation*}%
Let $\theta _{l}:=\mu _{L}\left( x:x\text{ is }\mu _{l}-\text{ typical}%
\right) $ and note that 
\begin{equation*}
\frac{\#A_{1}\left( N\right) \cap \hat{D}_{N,n(N)}}{\#A_{2}\left( N\right)
\cap \hat{D}_{N,n(N)}}\rightarrow \frac{\theta _{1}}{\theta _{2}}\text{ for }%
N\rightarrow \infty .
\end{equation*}%
Thus by gluing all the orbit segments of $f_{N}$ in the sets $A_{1}\left(
N\right) \cap \hat{D}_{N,n(N)}$ and $A_{2}\left( N\right) \cap \hat{D}%
_{N,n(N)}$ in such a way that they form one periodic orbit we obtain a
family of periodic orbits with quantum limit $\alpha _{1}\mu _{1}+\alpha
_{2}\mu _{2}$ where $\alpha _{l}=\frac{\theta _{l}}{\theta _{1}+\theta _{2}}%
. $ The $\alpha $ and $\beta -$ densities are just $\theta _{1}+\theta _{2}.$

It is easy to construct in the same spirit approximation families $f_{N}$
for any values $\alpha _{1}$ and $\alpha _{2}=1-$ $\alpha _{1}.$ Suppose
first that $\alpha _{1}<\frac{\theta _{1}}{\theta _{1}+\theta _{2}}$ and
hence $\alpha _{2}>\frac{\theta _{2}}{\theta _{1}+\theta _{2}}.$ Take in
each stack $S_{i}(n(N))$ with $i\in A_{1}\left( N\right) $ approximately $%
\frac{\alpha _{1}}{\alpha _{2}}$ of the internal orbit segments. The
function $n\left( N\right) $ is sufficiently slowly growing (\ref{width}) to
ensure that there are enough points in the discretisation set $\hat{D}%
_{N,n(N)}$ we can guarantee the convergence to $\frac{\alpha _{1}}{\alpha
_{2}}.$ Gluing these segments together with all the internal orbit segments
of $A_{2}\left( N\right) $ yields a single periodic orbit $\mathcal{O}%
^{\left( N\right) }$. The family of periodic orbits $\left\{ \mathcal{O}%
^{\left( N\right) }\right\} _{N}$ defines a quantum limit for the measure $%
\alpha _{1}\mu _{1}+\alpha _{2}\mu _{2}$ with $\alpha $ and $\beta -$
density $\frac{\alpha _{1}}{\alpha _{2}}\theta _{1}+\theta _{2}>0.$ The case 
$\alpha _{1}>\frac{\theta _{1}}{\theta _{1}+\theta _{2}}$ is analogous.

The proof of d) follows immediately by combining the arguments from parts b)
and c).
\end{proofof} 

\section{Examples}

\subsection{Interval Exchange Maps}

Consider a permutation $\pi $ of $\{1,2,\ldots ,n\}$ and a vector $\vec{v}%
=(v_{1},\ldots ,v_{n})$ such that $v_{i}>0$ for all $i$ and $%
\sum_{i=1}^{n}v_{i}=1$. Let $u_{0}=0$, $u_{i}=v_{1}+\dots +v_{i}$ and $%
\Delta _{i}=(u_{i-1},u_{i})$. The interval exchange transformation $T=T_{\pi
,\vec{v}}$, $T:[0,1]\rightarrow \lbrack 0,1]$ is the map that is an isometry
of each interval $\Delta _{i}$ which rearranges these intervals according to
the permutation $\pi $.

The Lebesgue measure is always an invariant measure for an IET. A typical
IET is uniquely ergodic, however there exist minimal, non uniquely ergodic
IETs. The first example of a minimal, non uniquely ergodic IET was given by
Keynes and Newton \cite{Keynes} and Keane \cite{keane}. The number of
ergodic invariant measures for a minimal IET on $m$ intervals is bounded by
the $\lfloor m/2 \rfloor$. \cite{katok1,veech}. The set of invariant measures always includes
absolutely continuous measures but can also include singular measures. It is
known that an interval exchange transformation on $m$ intervals is at most
of rank $m$, in particular it is a finite rank transformation (see for
example \cite{ferenczi}). In fact the typical IET is of rank 1 \cite{veech2}%
, although we will not use this fact. Thus we can apply Corollary \ref%
{cor:finiterank} to conclude:

\begin{enumerate}
\item {} any uniquely ergodic IET is quantum uniquely ergodic,

\item {} any minimal, non uniquely ergodic IET is not quantum uniquely
ergodic,

\item {} any absolutely continuous invariant measures appear as a positive
density quantum limit,

\item {} any singular ergodic invariant measure appears as a zero density
quantum limit.
\end{enumerate}

\subsection{The full shift}

Another example of a cutting and stacking transformation $f_{\mathcal{B}}$
that has $\mu _{L}$ as an ergodic invariant measure and admits further
singular measures $\mu $ such that $\mu (\mathbb{S})=0$ is given by the full
shift. Take any cutting and stacking model of the full two-sided shift on
two symbols with Bernoulli-measure $p_{0}=p_{1}=\frac{1}{2}$ (for details of
such models we refer to the book \cite{shields}). Note that although the
full shift has many periodic orbits the cutting and stacking model has none.
We remark that one could introduce some periodic orbits at the boundaries of
the subintervals but they would all sit or fall at singularity points and
hence do not appear as quantum limits, in other words there are no scars in
quantised cutting and stacking skew product mappings. By Proposition 3 all
other invariant measures of the full shift have no support on the
singularity set. Hence we can apply Theorem 3. It is interesting to note,
that the fractal-dimensions (box or Hausdorff dimension) of the singularity
set are rather large and that the upper and lower dimensions do not
coincide. A straightforward counting argument shows for instance that the
upper and lower box dimensions are in the open interval $\left( \frac{1}{2}%
,1\right) .$

\section{Comments and Conclusions}

We have shown in this paper that for a rather general class of dynamical
systems on the torus the variety of different invariant measures can be
recovered as quantum limits of the corresponding proper families of
quantised maps. The quantisation scheme here used is based on the one
introduced by \cite{mr}. For a discussion of alternative quantisation
procedures and a critical comparison we refer to the recent work \cite{Zel05}.

One of the main features in our systems is the presence of singularities. In
the quantisation procedure this provides enough freedom to obtain
eigenfunctions reflecting the typical orbit structure with respect to any non
atomic ergodic measure. It is an interesting
question whether our results are still valid in case the classical dynamical
system has no singularities. We conjecture that similar statements can be
obtained. For this it seems natural to replace the top-bottom gluing scheme
in the interval exchange approximating family by cutting and
``crossover-concatenation'' of touching
period orbits.

Concerning the quantisation of flows one might hope that a good
understanding of the associated quantised Poincare maps can guide one to a
deeper understanding of concrete features of eigenfunctions and spectrum. An
natural class of examples to study this questions are polygonal billiards.
In the case of rational polygons the associated Poincare maps for the
directional flow are interval exchange transformations which can be
quantised similar to the quantisation used in this paper. It would be
interesting to compare the results obtained that way with the semiclassical
properties of the direct flow quantisation via the billiard Hamiltonian.


\def\cprime{$'$}
\providecommand{\bysame}{\leavevmode\hbox to3em{\hrulefill}\thinspace}
\providecommand{\MR}{\relax\ifhmode\unskip\space\fi MR }
\providecommand{\MRhref}[2]{%
  \href{http://www.ams.org/mathscinet-getitem?mr=#1}{#2}
}
\providecommand{\href}[2]{#2}

\end{document}